\documentclass[11pt]{article}

\title{Strong solutions to singular SDEs and application to the Lennard-Jones potential.
}
\author{  D. Morale, G. Rui, S. Ugolini\\ \medskip Department of Mathematics, University of Milano, Italy. } 

\usepackage{cancel}
\usepackage{makeidx}         
\usepackage{graphicx}        
\usepackage{multicol}        
\usepackage[bottom]{footmisc}
\usepackage{newtxtext}       %
\usepackage{amsthm}
\usepackage[varvw]{newtxmath}       
\usepackage{todonotes}
\usepackage[font=normalsize,labelfont=bf]{caption}
\usepackage{geometry} 

\usepackage{algpseudocode}
\usepackage{url}

 \usepackage{enumitem} 
 \usepackage{chemformula,comment}
 \usepackage[
    backend=biber,
    style=numeric,
    sorting=nyt,
    maxbibnames=10,
  ]{biblatex}
\newtheorem{theorem}{Theorem}[section]
\newtheorem{proposition}[theorem]{Proposition} 

\newtheorem{lemma}{Lemma}[section]

\theoremstyle{definition}
\newtheorem{definition}[theorem]{Definition}

\newtheorem{example}{Example}[section]
 
\theoremstyle{remark}       
\newtheorem{remark}{Remark}[section]

\theoremstyle{definition} 
\newtheorem*{assumptionH}{Assumption H}
\newtheorem*{assumptionR}{Assumption R}

\begin{document}

\maketitle
\abstract{
We prove existence and  uniqueness of strong solutions to a large class of autonomous stochastic differential equations on an open domain, where the drift exhibits a singular behaviour at the boundary. The main result involves a drift composed of the gradient of a singular potential and an additional possibly singular force. In order to achieve the well-posedness of the model, we employ a probabilistic regularization approach.  Under suitable conditions, it is shown that  the explosion time of the solution process is infinite. The result is finally applied to the case of an interacting particle system subject to a Lennard-Jones potential, which is singular at the origin. 
}
\parindent=0pt

\medskip


\smallskip

\section{Introduction}

The main focus of this paper is the  investigation of the well-posedness of strong solutions to Brownian motion driven stochastic differential equations (SDEs) defined in an open set  $U\subset\mathbb R^d$, where the drift exhibits a singularity at the boundary $\partial U$. More precisely, we consider autonomous SDEs,  defined on a filtered probability space $\left(\Omega,\mathcal{F},(\mathcal{F}_{t})_{t\in [0,T]},\mathbb{P}\right)$, with $T>0$ a fixed time horizon, which,  for any $t\in [0,T]$, reads as
$$
dX_t = \left[-\nabla \Phi(X_t)+ \mu(X_t)\right] dt + \sigma(X_t) dW_t,
$$
 where \(\Phi:U\to\mathbb R\) is a singular potential and \(\mu:U\to\mathbb R^d\) is an additional locally Lipschitz perturbation, possibly singular near \(\partial U\). Here $\sigma: \mathbb{R}^d \to \mathbb{R}^{d\times d}$ is a regular diffusion coefficient and
 $W=\{W_t\}_{t\in[0,T]}$ is a $d$-dimensional Brownian motion.
The main result establishes the existence and pathwise uniqueness of a strong solution to the previous SDE till its explosion time, which never occurs at finite time.  In order to prove such well-posedness result, we propose a regularization approach, as in \cite{2015_Godinho_quininao,2016_Liu_Yang},  and  find proper sufficient conditions on the potential $\Phi$ which guarantee the almost sure existence, through a barrier estimate for the regularised processes. In particular, by proper regularity assumptions  upon the possible random  initial conditions, we establish that,  with probability one, the border $\partial U$ is almost surely not reached in finite time.  
The choice of the form of the drift is motivated by physical and applicative considerations. Furthermore the literature regarding the study of singular SDEs mainly deals with the case of gradient type drift (see, e.g., \cite{2003_albeverio_kondratiev_Rockner,2018_Dai_Pra,2005_krylov_rockner}).

\smallskip

 As a key application, in the second part of the paper we face a system of stochastic particles subject to Brownian noise and  pairwise interactions governed by a Lennard-Jones type potential  $V$, which is highly singular at the origin.  Indeed, the potential  $V$ has the following expression
\begin{eqnarray*}
    V(x) = \frac{A}{|x|^{\alpha}} - \frac{B}{|x|^{\beta}}, \qquad \alpha>\beta> 0.
\end{eqnarray*}  
The potential above models both attraction and repulsion  between two  particles: the first term represents a short-range strong repulsion to prevent particle overlap, while the second term represents a weaker, attractive force at intermediate range, able to model the Van der Waals forces. In the literature, in dimension $d=3$, the classical Lennard-Jones potential has exponents $\alpha=12, \beta =6$ and it is referred to as the $12-6$ law  \cite{Wales}.  

\smallskip

Let $N\in \mathbb N$, and consider an $N$-dimensional stochastic process $\{(X^1_N,...,X^N_t)\}_{t\in[0,T]}$, where 
 $X_t^i \in \mathbb{R}^d$ denotes the position of the $i$-th particle out of $N$ at time $t\in [0,T]$.  The evolution of the particle system is described by the following first-order system of SDEs, for any $ t\in [0,T]$ and for any $i=1,\ldots,N$
\begin{equation*}
\begin{split}
         dX_t^i &= -\frac{1}{N}\mathop{\sum_{j=1}^N}_{j\neq i} \nabla V \left(X_t^i-X_t^j\right) dt + \sigma dW_t^i,
\end{split}
\end{equation*}
 where $\sigma\in \mathbb R_+$ and $\{(W_{t}^{1},\ldots, W_{t}^{N})_{t\in[0,T]}\}$ is  a  family of $N$-dimensional $\mathbb R^d$-valued Brownian motion,  defined on the above filtered probability space. The overall force exerted on the $i$-th  particle is an empirical mean of the force $F(r)= -\nabla V(r)$ due to any other particle located at a distance $r$ from it.

\smallskip
Our interest in such kind of interactions is due to physical applications, such as recent specific applications in Cultural Heritage \cite{2018_Dai_Pra,javergaard2026hybrid,Rui_MACH_2021}. The classical Lennard-Jones force is anyway widely used in modelling, especially  in molecular dynamics  \cite{2020_Flandoli_Leocata_Ricci,2016_MZCJ}. However, at least to the knowledge of the authors, apart from a recent paper \cite{2025_GM}, there are no rigorous analytical results when the cited potential is used in stochastic models.

 \medskip

\emph{Related literature.}
The study of well-posedness problem for SDEs with singular drifts has garnered significant attention in recent years due to their applications in various research areas  such as physics, finance, and biology. Extensive research has been conducted on one-dimensional SDEs with singular drift terms. Notable works include comprehensive classifications of singularities and their impact on the existence and uniqueness of solutions in one-dimensional systems  \cite{2005_cherny}. The extension of these results to higher-dimensional SDEs introduces additional complexities. The literature on higher-dimensional cases often imposes stringent regularity conditions on the singularity, such as the integrability requirement that the drift coefficient belongs to $L^p_{loc}$ for sufficiently large $p$. This ensures that the singularity is not too severe and allows for the application of certain analytical techniques \cite{2001_Gyongy_martinez,2021_Kinzebulatov_Semenov,1985_Takanobu,2023_Yang_Zhang,2005_Zhang}.   We also refer to \cite{2005_krylov_rockner} and \cite{2003_albeverio_kondratiev_Rockner} in which the authors prove  existence and uniqueness of strong solutions to stochastic equations in open domains  with unit diffusion and singular time-dependent drift up to the explosion time, the first giving integrability conditions under which the explosion time is infinite and the distributions of the solution have sub-Gaussian tails and, for the second, establishing strong Feller properties for a class of distorted Brownian motions. With respect to these papers we consider a more general drift, even though it is time independent, allowing for drifts not in gradient form and also with random initial conditions.
In the last decade, there has been notable progress in studying SDEs with specific types of singular drifts, including Poisson and Coulomb interactions \cite{2015_Godinho_quininao,2016_Liu_Yang,2019_Liu_Yang}. The interested reader may also refer to \cite{2020_Nam,2014_Beck_Flandoli_Gubinelli_Maurelli} for different integrability conditions.

\smallskip

At the best of our knowledge, there is no comprehensive study addressing SDEs with the not smoothed Lennard-Jones potential. Recently, a study via a regularization approach at a mesoscale is considered in \cite{2025_GM}; on one hand, the authors  deal with a McKean-Vlasov SDE with the drift given by a functional of the Lennard-Jones potential only and study the right integrability properties of the kernel in order to  establish the well-posedness; on the other hand, they consider an associated interacting particle system regularizing the drift where the particles interact via the force at a mesoscale of an appropriate order. The approach we propose here is different: indeed, we consider a more probabilistic method and  our first well-posedness result is quite more general in the following sense. We study a general SDE with an additive drift which includes a gradient term, depending on a potential which is singular at the origin, and with a second term that may be general with a possible singularity of order less than the order of the first one. Given regularity conditions upon the infinitesimal generator of the stochastic process, we prove that there exists a unique global strong
solution since there is no blow-up in a finite time. We address a proper rescaling procedure as in  \cite{2016_Liu_Yang}. With respect to the latter paper in which only a specific repulsive interaction is considered,  we allow a much stronger singularity and we include both a repulsive and attractive term. The application to the Lennard-Jones case is straightforward. We prove that it allows the correct control on the infinitesimal generator, with a right balance between the repulsive and the aggregation forces.
\medskip

The paper is organized as follows. In Section 2 the main result of the paper, namely a well-posedness result for autonomous Brownian motion driven SDEs with singular drift, is proposed  in a general setting. Section 3 is devoted  to the specific, but relevant, case in which the SDE describes an interacting particle system where the interaction force is derived from the Lennard-Jones potential.

\section{Well-posedness of an SDE with an additive singular drift}\label{sec:wellposedness_LJ}

Let  \(\left(\Omega,\mathcal{F},\mathbb F=(\mathcal{F}_{t})_{t\in [0,T]},\mathbb{P}\right)\) be a filtered probability space satisfying the usual conditions, and
let \(T>0\) be a fixed time horizon. We denote by \(U\subset\mathbb R^d\) an open set  with boundary \(\partial U\).
 
The main result of the section concerns the well-posedness of a class of SDEs whose drift contains a singular gradient component together with an additional perturbation. We then briefly consider a more general case with drift not necessarily in gradient form.

\subsection{Singular gradient drift with an additional perturbation}

We consider a class  of  SDEs  on  $U$, whose drift consists of two components: a gradient term generated by a singular potential  $\Phi:U\to \mathbb{R}$, and an additional vector field $\mu:U\to \mathbb{R}^d$, which is not necessarily of gradient type and may itself become singular as the boundary of \(U\) is approached.

More precisely, let $X= \{X_t\}_{t\in [0,T]}$ satisfy, up to its lifetime in \(U\),
\begin{equation}\label{eq:SDE_gradientform}
dX_t = -\nabla \Phi(X_t)dt + \mu(X_t)dt + \sigma(X_t) dW_t,
\end{equation}
where $W=\{W_t\}_{t\ge 0}$ is an $\mathbb F$-adapted $d$-dimensional Brownian motion and $\sigma: \mathbb R^{d} \to \mathbb R^{d\times d}$ denotes the diffusion coefficient.

We now introduce the assumptions under which global well--posedness of \eqref{eq:SDE_gradientform} will be established.

\begin{assumptionH}
Let
\(
\Phi\in C^2(U),
\,
\mu\in C_{\mathrm{loc}}^{0,1}(U;\mathbb R^d)
\), and
assume that the following conditions hold.
\begin{itemize}
    \item [$H_1$.]\textit{Boundedness from below:}  There exists a constant $\delta \ge 0$ such that
    \[
        \Phi(x) \geq -\delta, \qquad x\in U.
    \]
    \item[$H_2$.] \textit{Blow up at the boundary:} The potential diverges when the boundary is approached:
    \[
       \lim\limits_{d(x,\partial U)\to 0} \Phi(x)  = + \infty.
    \]
    \item[$H_3$.] \textit{Lyapunov condition:} 
    There exists a constant \( \eta > 0 \) such that, for every \( x \in U \),
    \begin{equation}\label{Hp:H3}
        - \left|\nabla\Phi(x)\right|^2
        + \mu(x) \cdot \nabla\Phi(x)
        + \frac{1}{2} \operatorname{Tr}\left[ \sigma(x)\sigma(x)^\top \nabla^2\Phi(x)
        \right] \leq \eta.
    \end{equation}
    \item [$H_4$.] \textit{Diffusion regularity: }
    The diffusion coefficient $\sigma$ is globally Lipschitz and bounded.
    \end{itemize}
\end{assumptionH}

In order to prove global existence and uniqueness for \eqref{eq:SDE_gradientform}, we employ a regularisation-and-limit strategy.
The singular coefficients are modified only in a neighbourhood of the boundary \(\partial U\), while the original dynamics is preserved away from the singular region. We then show that the probability that the regularised process reaches this region before any fixed finite time tends to zero as the regularisation parameter vanishes. This approach is inspired by the regularisation arguments developed in \cite{2016_Liu_Yang} and allows us to treat a broader class of singular gradient-type drifts.

For $\varepsilon > 0$, define the inner domain
\[
U_\varepsilon := \left\{x \in U : d(x, \partial U)> \, \varepsilon \right\}.
\]
We introduce the following assumption on the regularised coefficients.
\begin{assumptionR}[Regular approximation of the drift] 
Let $\Phi$ and $\mu$ satisfy Assumption \emph{H}. For every \(\varepsilon>0\), assume that there exist  $\Phi^\varepsilon \in C^2(\mathbb{R}^d, \mathbb{R})$ and $\mu^\varepsilon \in C(\mathbb{R}^d, \mathbb{R}^d)$ such that:
    \begin{itemize}
    \item[$R_1$.]
    \textit{Consistency away from the singular boundary.}
    $\Phi^\varepsilon(x) = \Phi(x)$ and $\mu^\varepsilon(x) = \mu(x)$ for every $x \in \overline{U_\varepsilon}$.
    
    \item[$R_2$.]
    \textit{Uniform lower bound.} Given the constant \(\delta\) from Assumption \(H_1\), the regularised potentials satisfy \[ \Phi^\varepsilon(x)\geq-\delta, \qquad x\in\mathbb R^d. \]

    \item[$R_3$.]
    \textit{Global regularity.}
    The vector fields $\nabla \Phi^\varepsilon$ and $\mu^\varepsilon$  are   globally Lipschitz.
\end{itemize}
\end{assumptionR}

\begin{definition}[regularised process]
  Let $\Phi,\mu $ satisfy Assumption H, and let \(\Phi^\varepsilon,\mu^\varepsilon\) be the corresponding regularised versions as in Assumption R.  Define the regularised process $X^\varepsilon=\{X_t^\varepsilon\}_{t\in[0,2T]}$ as the solution of
\begin{align}\label{eq:SDE_gradientform_regularised}
    dX_t^\varepsilon = -\nabla\Phi^\varepsilon(X_t^\varepsilon)dt + \mu^\varepsilon(X_t^\varepsilon)dt + \sigma(X_t^\varepsilon) dW_t, \qquad X_0^\varepsilon = X_0 \in U.
\end{align}
\end{definition}

\begin{proposition}\label{prop:existence_reg_system}
Suppose that Assumptions \emph{H} and \emph{R} are  satisfied. Then, for every \(\varepsilon >0 \), equation \eqref{eq:SDE_gradientform_regularised} admits a pathwise unique global strong solution. 
\end{proposition}
\begin{proof}
By \(R_3\) and \(H_4\), all coefficients of \eqref{eq:SDE_gradientform_regularised} are globally Lipschitz and hence have at most linear growth. Standard SDE theory therefore yields a pathwise unique global strong solution.  
\end{proof}

To control approaches to the singular boundary, for every
\(\varepsilon>0\) we introduce the first exit time of the regularised
process \(X^\varepsilon\) from the inner domain \(U_\varepsilon\),
truncated at \(2T\):
\begin{equation}\label{def:collisiontime_border_regularised}
    \tau_\varepsilon
    :=
    \inf\left\{t\geq 0:
    X_t^\varepsilon\notin U_\varepsilon\right\}
    \wedge 2T.
\end{equation}
The truncation at \(2T\) ensures that \(\tau_\varepsilon\) is finite, while leaving unchanged the event \(\{\tau_\varepsilon\leq T\}\) considered below.

The following is the main result of this chapter.

\begin{theorem}[Existence and uniqueness under singular drifts]\label{thm:mainthm_existence_uniqueness}
    Let Assumptions \textit{H} and \emph{R} be satisfied. Let \(W=(W_t)_{t\ge0}\) be a \(d\)-dimensional \(\mathbb F\)-Brownian motion  on the filtered probability space $\left(\Omega, \mathcal{F}, \mathbb{F}, \mathbb{P}\right)$. Suppose the random initial condition $X_0$ is $\mathcal F_0$-measurable, is \(U\)-valued almost surely,
and satisfies
\begin{equation}
\mathbb{E}\left[|X_0|^2\right] < \infty, \qquad 
\mathbb{E} \big[\Phi(X_0)\big] < \infty.
  \label{eq:mainthm_existence_uniqueness_initial_conditions}
\end{equation}
Then equation \eqref{eq:SDE_gradientform}  admits a pathwise unique global strong solution on any finite time interval $[0,T]$. In particular, the singularity on  $\partial U$ is almost surely never reached in finite time.
\end{theorem}
\begin{proof}
Our aim is to construct a strong solution \(X\) of
\eqref{eq:SDE_gradientform} and to show that it does not reach the singular boundary in finite time. 
To this end, we first consider the globally defined regularised processes \(X^\varepsilon\) and their corresponding exit times \(\tau_\varepsilon\). We prove that the probability that a
regularised process reaches the singular region before time \(T\) vanishes as \(\varepsilon\downarrow0\).

We first establish that, for every fixed \(T>0\),
\begin{equation}\label{eq:prob_stopping_time-varepsilon}
\lim_{\varepsilon\downarrow0}
\mathbb P(\tau_\varepsilon\leq T)=0.
\end{equation}
We then restrict to a countable sequence
\(\varepsilon_n\downarrow0\) and use the consistency of the
corresponding regularised solutions to construct a strong solution of \eqref{eq:SDE_gradientform} on \([0,T]\). Since \(T>0\) is
arbitrary, this yields a global strong solution and the claimed non-attainment of \(\partial U\). 
The estimate \eqref{eq:prob_stopping_time-varepsilon} is established in the steps below, following ideas from
\cite{2015_Godinho_quininao,2016_Liu_Yang}.

\smallskip
\noindent\emph{Step 1. Martingale decomposition.}
For brevity, set
\[
\nabla\Phi_t^\varepsilon
:=\nabla\Phi^\varepsilon(X_t^\varepsilon),
\qquad
\mu_t^\varepsilon:=\mu^\varepsilon(X_t^\varepsilon),
\qquad
\sigma_t:=\sigma(X_t^\varepsilon).
\]
By applying It\^o's formula to $\Phi^\varepsilon_{t\wedge\tau_\varepsilon}$ on $[0,t\wedge\tau_\varepsilon]$, we obtain
\begin{equation*}
\Phi^\varepsilon_{t\wedge\tau_\varepsilon}
    =  \Phi^\varepsilon_0 - \int_0^{t\wedge\tau_\varepsilon}\left(|\nabla\Phi_s^\varepsilon|^2 -\nabla\Phi_s^\varepsilon\cdot\mu_s^\varepsilon - \frac{1}{2} \operatorname{Tr}\left[ \sigma_s^\top \nabla^2 \Phi^\varepsilon_s\, \sigma_s \right] \right)ds + M_{t\wedge\tau_\varepsilon},
\end{equation*}
where $M_{t}$ is the local martingale
\begin{equation*}
    M_{t} = \int_0^{t} \nabla\Phi_s^\varepsilon \cdot \sigma_s dW_s.
\end{equation*}
The process $ M_{t} $ is actually an $\mathcal{F}_t$-square-integrable martingale. Indeed, under Assumption $H_4$ the diffusion coefficient is bounded, and under Assumption $R_3$ the regularised force $\nabla\Phi^\varepsilon$ has at most linear growth. Therefore,
\begin{align}\label{eq:gradient_estimate}
    \int_0^T \mathbb{E}\left[|\nabla\Phi_s^\varepsilon\cdot \sigma_s|^2\right] ds &
    \leq c_\varepsilon\left( T + \int_0^T \mathbb{E}\left[|X_s^\varepsilon|^2\right] ds\right),
\end{align}
and by \eqref{eq:SDE_gradientform_regularised}, one gets
\begin{align*}
    \mathbb{E}\left[|X_t^\varepsilon|^2\right] &=
    \mathbb{E}\left[\left|X_0 - \int_0^t \nabla\Phi_s^\varepsilon\, ds
    + \int_0^t \mu_s^\varepsilon ds
    + \int_0^t \sigma_s dW_s\right|^2 \right]  \\
    &\leq 4\mathbb{E}\left[|X_0|^2\right]
    + 4\mathbb{E}\left[\left|\int_0^t  \nabla\Phi_s^\varepsilon ds \right|^2 \right] 
    +4\mathbb{E}\left[\left|\int_0^t \mu_s^\varepsilon ds \right|^2 \right] \\
    &\qquad
    +4\mathbb{E}\left[
    \left|
    \int_0^t\sigma_s\,dW_s
    \right|^2
    \right].
\end{align*}
By It\^o's isometry and the boundedness of \(\sigma\),
\[
\mathbb{E}\left[
\left|
\int_0^t \sigma_s\,dW_s
\right|^2
\right]
\leq C\,t,
\]
for some constant \(C>0\).
By the H\"older's inequality, we obtain
\begin{align*}
    \mathbb{E}\left[\left|\int_0^t  \nabla\Phi_s^\varepsilon ds \right|^2 \right] & \leq  
    t\mathbb{E}\left[\int_0^t \left| \nabla\Phi_s^\varepsilon \right|^2\, ds\right] \leq c_{\varepsilon,T}\left(1 +  \int_0^t \mathbb{E}[|X_s^\varepsilon|^2]\, ds\right).
\end{align*}
Similarly, the same estimate holds for the $\mu^\varepsilon$ term by linear growth.
Hence, by combining all terms, we obtain the following inequality:
\[
    \mathbb{E}\left[\left|X_t^\varepsilon\right|^2\right]\leq C_{\varepsilon,T}+ C_{\varepsilon,T}\int_0^t \mathbb{E}\left[\left|X_s^\varepsilon\right|^2\right] ds,
\]
which, by Gr\"onwall's lemma, gives the following  bound, for all $t\in [0,T]$,
\begin{equation}\label{eq:exp_X^2_varepsilon}
    \mathbb{E}\left[|X_t^\varepsilon|^2\right] \leq C_{\varepsilon,T,d}.
\end{equation}
By It\^o's isometry for $M_{t}$ and by \eqref{eq:gradient_estimate} and \eqref{eq:exp_X^2_varepsilon}, we obtain that $M_{t}$ is indeed a square-integrable martingale. This completes Step 1.

\smallskip
\noindent \emph{Step 2. Bounds on the martingale.}
We derive pathwise lower bounds on the infimum and supremum of the stopped martingale $M_{t\wedge \tau_\varepsilon}$. 
More precisely, on the event \({X_0\in U_\varepsilon}\), we show that
\begin{align}
    \inf_{t\in[0,T]} M_{t\wedge\tau_\varepsilon} &\geq - \Phi_0 - \delta - \eta T, \label{eq:EstimateMinf}\\
    \underset{t\in[0,T]}{\sup}M_{t\wedge \tau_\varepsilon} &\geq  \underset{t\in[0,T]}{\sup}\Phi_{t\wedge\tau_\varepsilon}^\varepsilon - \Phi_0 - \eta T. \label{eq:EstimateMsup}
\end{align}
Here \(\eta\) is the constant appearing in Assumption \(H_3\), independent of \(\varepsilon\), and \(\delta\) is the lower-bound constant from Assumption \(H_1\).

Indeed, from the It\^o decomposition obtained in Step 1,
\begin{equation*}
M_{t\wedge\tau_\varepsilon}= \Phi^\varepsilon_{t\wedge\tau_\varepsilon}
    - \Phi^\varepsilon_0 + \int_0^{t\wedge\tau_\varepsilon}\left(|\nabla\Phi_s^\varepsilon|^2  - \nabla\Phi_s^\varepsilon\cdot\mu_s^\varepsilon - \frac{1}{2} \operatorname{Tr}\left[ \sigma_s^\top \nabla^2 \Phi^\varepsilon_s\, \sigma_s \right]\right) ds,
\end{equation*}
which gives by \eqref{Hp:H3} in Assumption $H_3$, that
\[
    M_{t\wedge\tau_\varepsilon}\geq \Phi^\varepsilon_{t\wedge\tau_\varepsilon}
    - \Phi_0 - \eta T  .
\]
Finally, by Assumption $R_2$, inequalities \eqref{eq:EstimateMinf} and \eqref{eq:EstimateMsup} are  obtained.

\smallskip
\noindent \emph{Step 3. Collision probability estimate.}
We are going to estimate the probability that a collision occurs before any fixed time $T$, namely the quantity $\mathbb{P}(\tau_\varepsilon \leq T)$.

Define $f(\varepsilon)
:=
\inf\left\{
\Phi(x):
x\in U,\;
d(x,\partial U)\leq\varepsilon
\right\},$  for $\varepsilon >0$, so that by Assumption $H_2$
\begin{equation}\label{eq:f(varepsilon)_to_zero}
    \lim_{\varepsilon \to 0} f(\varepsilon) = +\infty.
\end{equation}
Since the estimates obtained in Step 2 hold on the event \(\{X_0\in U_\varepsilon\}\), we first decompose
\[
\mathbb P(\tau_\varepsilon\leq T)
\leq
\mathbb P(X_0\notin U_\varepsilon)
+
\mathbb P(\tau_\varepsilon\leq T,\ X_0\in U_\varepsilon).
\]
Moreover, if \(X_0\notin U_\varepsilon\), then \(d(X_0,\partial U)\leq\varepsilon\) and \(
\Phi(X_0)\geq f(\varepsilon). \)
Since \(f(\varepsilon)\to+\infty\), for all sufficiently small
\(\varepsilon\) we have
\(f(\varepsilon)+\delta>0\).
Hence
\[
\mathbb P(X_0\notin U_\varepsilon)
\leq
\mathbb P\bigl(\Phi(X_0)+\delta\geq f(\varepsilon)+\delta\bigr)
\leq
\frac{\mathbb E[\Phi(X_0)+\delta]}
{f(\varepsilon)+\delta}
\longrightarrow0
\qquad\text{as }\varepsilon\downarrow0.
\]
It remains therefore to estimate
\(\mathbb P(\tau_\varepsilon\leq T,\ X_0\in U_\varepsilon)\).

If \(\tau_\varepsilon\leq T\) and \(X_0\in U_\varepsilon\), then
\[
\sup_{t\in[0,T]}\Phi^\varepsilon_{t\wedge\tau_\varepsilon}
\geq
\Phi_{\tau_\varepsilon}^\varepsilon.
\]
Using the estimates obtained in Step 2, we deduce
\[
\begin{aligned}
\{\tau_\varepsilon\leq T,\ X_0\in U_\varepsilon\}
&\subseteq
\left\{
\sup_{t\in[0,T]}\Phi^\varepsilon_{t\wedge\tau_\varepsilon}
\geq
\Phi_{\tau_\varepsilon}^\varepsilon
\right\}
\\
&\subseteq
\left\{
\sup_{t\in[0,T]}M_{t\wedge\tau_\varepsilon}
\geq
\Phi_{\tau_\varepsilon}^\varepsilon-\Phi_0-\eta T
,
\inf_{t\in[0,T]}M_{t\wedge\tau_\varepsilon}
\geq
-\Phi_0-\delta-\eta T
\right\}.
\end{aligned}
\]
On the event \(\{\tau_\varepsilon\leq T,\ X_0\in U_\varepsilon\}\), the truncation at \(2T\) is inactive, so that \(\tau_\varepsilon\) is the first exit time from \(U_\varepsilon\). Moreover, since \(X^\varepsilon\) has continuous paths we have that \(d(X^\varepsilon_{\tau_\varepsilon},\partial U)=\varepsilon\), and
\[
\Phi^\varepsilon_{\tau_\varepsilon}\geq f(\varepsilon).
\]
Therefore
\begin{align*}
    \left\{
    \tau_\varepsilon\leq T,\ X_0\in U_\varepsilon\right\} 
    &\subseteq
    \left\{\underset{t\in[0,T]}{\sup}M_{t\wedge \tau_\varepsilon} \geq f(\varepsilon) - \Phi_0- \eta T,\underset{t\in[0,T]}{\inf}M_{t\wedge \tau_\varepsilon} \geq - \Phi_0 -\delta - \eta T \right\}\\
    &:=A \cap B.
\end{align*}
Fix \(R>0\). Then
\[
\begin{aligned}
\mathbb P(\tau_\varepsilon\leq T,\ X_0\in U_\varepsilon)
&\leq
\mathbb P\bigl(A\cap B\cap\{\Phi_0+\delta<R\}\bigr)
+
\mathbb P(\Phi_0+\delta\geq R)
\\
&\leq
\mathbb P\left(
\sup_{t\in[0,T]}M_{t\wedge\tau_\varepsilon}
> f(\varepsilon)-R-\eta T,\,
\inf_{t\in[0,T]}M_{t\wedge\tau_\varepsilon}
> -R-\eta T
\right)
\\
&\qquad
+\mathbb P(\Phi_0+\delta\geq R)
=:P_1+P_2.
\end{aligned}
\]
By Markov's inequality and assumption \eqref{eq:mainthm_existence_uniqueness_initial_conditions}
\begin{equation*}
    P_2=\mathbb{P}\left(\Phi_0+\delta \geq R\right)\leq \frac{\mathbb{E}[\Phi_0+\delta]}{R} \leq \frac{C}{R}.
\end{equation*}
On the other hand, set
\[
a:=R+\eta T,
\qquad
u:=f(\varepsilon)-R-\eta T,
\]
and assume \(\varepsilon\) is sufficiently small so that \(u>0\). Define the stopping time
\[
T_{-a}
:=
\inf\left\{
t\geq0:
M_{t\wedge\tau_\varepsilon}\leq -a
\right\}.
\]
Consider the stopped process
\[
Y_t
:=
a+M_{t\wedge\tau_\varepsilon\wedge T_{-a}}.
\]
Then \(Y\) is a non-negative martingale and \(Y_0=a\).

Moreover, on the event defining \(P_1\), the martingale reaches the level \(u\) before reaching \(-a\). Hence
\[
P_1
\leq
\mathbb P\left(
\sup_{t\in[0,T]}Y_t\geq a+u
\right).
\]
By Doob's maximal inequality for non-negative martingales (see, e.g., \cite{RogersWilliams}),
\[
P_1
\leq
\frac{\mathbb E[Y_0]}{a+u}
=
\frac{a}{a+u}
=
\frac{R+\eta T}{f(\varepsilon)}.
\]
Gathering together the different terms, we obtain that, for any $R>0$
\[
    \mathbb{P}\left(\tau_\varepsilon\leq T\right) \leq
    \frac{\mathbb E[\Phi(X_0)+\delta]}
{f(\varepsilon)+\delta} + \frac{C}{R} + \frac{R+\eta T}{f(\varepsilon)}.
\]
We now choose $R = \sqrt{f(\varepsilon)}$. From \eqref{eq:f(varepsilon)_to_zero}, it follows that
\[
    \mathbb{P}(\tau_\varepsilon \leq T)\rightarrow 0\qquad \text{as} \quad \varepsilon \downarrow 0.
\]

\smallskip
\noindent
\emph{Step 4. Existence and uniqueness of a global strong solution.} As mentioned at the beginning of the proof, now we aim at constructing  a global strong solution to equation \eqref{eq:SDE_gradientform} as the limit of the regularised processes \( \{X^\varepsilon_t\}_{\varepsilon>0}\), thus establishing its pathwise uniqueness. 

We fix a countable subsequence \(\varepsilon_n\downarrow0\). 
By  Proposition \ref{prop:existence_reg_system}, for any such $\varepsilon_n$, the regularised equation \eqref{eq:SDE_gradientform_regularised} admits a pathwise unique strong solution $\left\{X_t^n\right\}_{t\in [0,2T]}$. We set the corresponding exit time as
\(
\tau_n:= \tau_{\varepsilon_n}.
\)
By the estimate just proved in Step 3, we have
\(
\mathbb P(\tau_n\le T)\longrightarrow0.
\)

Let \(m\geq n\). Since
\(
\varepsilon_m\leq\varepsilon_n,
\)
we have
\(
U_{\varepsilon_n}\subset U_{\varepsilon_m}.
\)
By Assumption \(R_1\), the coefficients of the corresponding regularised equations  coincide on \(U_{\varepsilon_n}\). Hence, by pathwise uniqueness, \(X^m\) and \(X^n\) coincide up to the first time either process exits \(U_{\varepsilon_n}\).
Since the two processes have continuous paths and agree up to that time, they exit \(U_{\varepsilon_n}\) simultaneously. Therefore, almost surely,
\[
X^m_t=X^n_t,
\qquad 0\leq t\leq\tau_n.
\]
Moreover, since \(U_{\varepsilon_n}\subset U_{\varepsilon_m}\), it follows that  \(\tau_n\leq\tau_m\) almost surely.

Since the set of pairs
\(
\{(m,n)\in\mathbb N^2:m\geq n\}
\)
is countable, there exists an event \(A\in\mathcal F\), with
\(\mathbb P(A)=1\), on which
\begin{equation}\label{eq:consistency}
X^m_t=X^n_t,
\qquad 0\leq t\leq\tau_n,
\end{equation}
for every \(m\geq n\), simultaneously.
In particular, on \(A\), the sequence \((\tau_n)_n\) is
non-decreasing, and
\begin{equation*}
\mathbb{P} \left( \lim_{n \rightarrow + \infty} \tau_n > T \right) = \lim_{n \rightarrow + \infty} \mathbb{P}(\tau_n > T) = 1. 
\end{equation*}

Therefore, almost surely there exists an index
\(n_0=n_0(\omega)\) such that
\[
\tau_n>T
\qquad\text{for every }n\geq n_0.
\]
Hence, by \eqref{eq:consistency}, for \(m\geq n\geq n_0\), on \(A\)
\[
X^m_t=X^n_t,
\qquad t\in[0,T].
\]
Thus the sequence \((X^n)_n\) is almost surely eventually
constant on the whole interval \([0,T]\).

We may therefore define
\[
X_t:=\lim_{n\to\infty}X^n_t,
\qquad t\in[0,T].
\]
The process \(X_t\) is \(\mathcal F_t\)-measurable,
being the almost-sure limit of the
\(\mathcal F_t\)-measurable random variables \(X^n_t\). Similarly, \(X\) has continuous paths, since the sequence \((X^n)_n\) is eventually constant on the whole
interval \([0,T]\), and each \(X^n\) has continuous paths.

We next verify that \(X\) satisfies the original equation.
For every fixed \(n\), passing to the limit as  \(m\to\infty\) in \eqref{eq:consistency} gives
\[
X_{t\wedge\tau_n}
=
X^n_{t\wedge\tau_n},
\qquad t\in[0,T],
\]
almost surely. By the definition of \(\tau_n\), \(X^n_s\in U_{\varepsilon_n}\) for \(s<\tau_n\). Hence, by \(R_1\), the regularised coefficients agree with the original ones up to \(\tau_n\). Consequently,
\[
X_{t\wedge\tau_n}
=
X_0
-
\int_0^{t\wedge\tau_n}
\nabla\Phi(X_s)\,ds
+
\int_0^{t\wedge\tau_n}
\mu(X_s)\,ds
+
\int_0^{t\wedge\tau_n}
\sigma(X_s)\,dW_s,
\qquad t\in[0,T],
\]
almost surely.

Since, almost surely, \(\tau_n>T\) for all sufficiently large
\(n\), the preceding identity yields
\[
X_t
=
X_0
-
\int_0^t\nabla\Phi(X_s)\,ds
+
\int_0^t\mu(X_s)\,ds
+
\int_0^t\sigma(X_s)\,dW_s,
\qquad t\in[0,T].
\]
Thus \(X\) is a strong solution of
\eqref{eq:SDE_gradientform} on \([0,T]\).

Pathwise uniqueness follows from that of the regularised equations.
Let \(\widetilde X\) be another strong solution on \([0,T]\), driven
by the same Brownian motion and with the same initial condition, and set
$$
\widetilde\tau_n
:=
\inf\{t\in[0,T]:\widetilde X_t\notin U_{\varepsilon_n}\},
$$
with the convention \(\inf\varnothing=\infty\).

By \(R_1\), \(X\) and \(\widetilde X\) solve the same
regularised equation up to \(\tau_n\wedge\widetilde\tau_n\), so pathwise uniqueness gives
$$
X_{t\wedge\tau_n\wedge\widetilde\tau_n}
=
\widetilde X_{t\wedge\tau_n\wedge\widetilde\tau_n},
\qquad t\in[0,T].
$$
Since \(\tau_n>T\) for all sufficiently large \(n\), almost surely, \(X\) does not exit \(U_{\varepsilon_n}\) on \([0,T]\) for such \(n\).

Since both processes have continuous paths, they exit
\(U_{\varepsilon_n}\) at the same time, so that
$$
\tau_n=\widetilde\tau_n
\qquad\text{almost surely}.
$$
Since \(\tau_n>T\) for all sufficiently large \(n\), almost surely, we conclude that
$$
X_t=\widetilde X_t,\qquad t\in[0,T].
$$

Finally, applying the preceding construction on \([0,k]\),
\(k\in\mathbb N\), and using pathwise uniqueness on overlapping
intervals, we obtain a pathwise unique global strong solution
\(X\).

To conclude, define the boundary hitting time
$$
\tau:=\inf\{t\geq0:X_t\in\partial U\},
$$
with the convention \(\inf\varnothing=\infty\). Since, for every \(k\in\mathbb N\), the solution remains in \(U\) throughout \([0,k]\) almost surely, we have
$$
\mathbb P(\tau>k)=1,
\qquad k\in\mathbb N.
$$
Therefore,
$$
\mathbb P(\tau=\infty)
=
\mathbb P\left(\bigcap_{k\in\mathbb N}\{\tau>k\}\right)
=1.
$$
Hence the singular boundary \(\partial U\) is almost surely never reached in finite time.

\end{proof}

\begin{remark}
The assumptions on the diffusion coefficient \(\sigma\) are imposed here for simplicity and are not expected to be optimal. We expect that the argument can be extended to diffusion coefficients with more general growth, and possibly to coefficients that become singular near \(\partial U\), provided suitable moment estimates and regularised approximations are available and the Lyapunov condition is adapted accordingly. We do not pursue these extensions here.
\end{remark}

\subsection{A general drift case}
The previous result can be extended to the case where the drift is not necessarily of gradient type, but still satisfies an appropriate Lyapunov-type condition.

\begin{proposition}\label{prop:general_drift}
Let \(U\subseteq\mathbb R^d\) be open and let
\(b:U\to\mathbb R^d\) be
locally Lipschitz, possibly singular as
\(\partial U\) is approached. Assume that
there exists \(\Psi\in C^2(U)\) such that:
\begin{itemize}
\item[i)]
\emph{Boundedness from below}: there exists \(\delta\geq0\) such that \(\Psi(x)\geq-\delta\) for \(x\in U.
\)
\item[ii)]
\emph{Blow-up near the boundary}:
\(
\Psi(x)\to+\infty
\, \text{as }d(x,\partial U)\to0.
\)
\item[iii)]
\emph{Regularity condition}: there exists a constant \(\eta>0\) such that for all \(x\in U\),
\[
b(x)\cdot\nabla\Psi(x)
+
\frac{1}{2}
\operatorname{Tr}
\left[
\sigma(x)\sigma(x)^\top
\nabla^2\Psi(x)
\right]
\leq \eta.
\]
\item[iv)]
\emph{Diffusion regularity}: the diffusion coefficient
\(\sigma:\mathbb R^d\to\mathbb R^{d\times d}\)
is globally Lipschitz and bounded.
\item[v)]
\emph{Regularisation}: for every \(\varepsilon>0\) there exists a regularisation
\(
b^\varepsilon:\mathbb R^d\to\mathbb R^d, \,  \Psi^\varepsilon\in C^2(\mathbb R^d)
\)
such that:
\begin{enumerate}
\item
\( b^\varepsilon(x)=b(x),\Psi^\varepsilon(x)=\Psi(x),
\qquad x\in \overline{U_\varepsilon};
\)
\item
 \(b^\varepsilon\) is globally Lipschitz,
\(\nabla\Psi^\varepsilon\) has at most linear growth;
\item
\(\Psi^\varepsilon(x)\geq-\delta,
\qquad x\in\mathbb R^d,\)
\end{enumerate}
\end{itemize}

Then, the SDE
\[
dX_t
=
b(X_t)\,dt
+
\sigma(X_t)\,dW_t
\]
admits a pathwise unique global strong solution, provided that the \(U\)-valued random initial condition \(X_0\) satisfies
\[
\mathbb{E}\left[|X_0|^2\right]<\infty,
\qquad
\mathbb{E}\left[\Psi(X_0)\right]<\infty.
\]
\end{proposition}
\begin{proof}
The proof follows the same structure as the proof of Theorem~\ref{thm:mainthm_existence_uniqueness}.

For every \(\varepsilon>0\), let \(X^\varepsilon\) be the global strong solution of the regularised equation
\[
dX_t^\varepsilon
=
b^\varepsilon(X_t^\varepsilon)\,dt
+
\sigma(X_t^\varepsilon)\,dW_t.
\]
Applying It\^o's formula to the scalar process
\(\Psi^\varepsilon(X_t^\varepsilon)\), stopped at the exit time
\(\tau_\varepsilon\) defined analogously to
\eqref{def:collisiontime_border_regularised}, gives
$$
\Psi^\varepsilon(X_{t\wedge\tau_\varepsilon}^\varepsilon)
=
\Psi^\varepsilon(X_0)
+
\int_0^{t\wedge\tau_\varepsilon}
\left[
b^\varepsilon(X_s^\varepsilon)\cdot\nabla\Psi^\varepsilon(X_s^\varepsilon)
+
\frac12
\operatorname{Tr}
\left(
\sigma\sigma^\top\nabla^2\Psi^\varepsilon
\right)(X_s^\varepsilon)
\right]ds
+
M_{t\wedge\tau_\varepsilon},
$$
where \(M\) is the corresponding local martingale.

By the consistency of the regularised coefficients on
\(\overline{U_\varepsilon}\), the integrand coincides, up to \(\tau_\varepsilon\), with
$$
b(X_s^\varepsilon)\cdot\nabla\Psi(X_s^\varepsilon)
+
\frac12
\operatorname{Tr}
\left(
\sigma\sigma^\top\nabla^2\Psi
\right)(X_s^\varepsilon).
$$
By assumption iii), the drift term is bounded above by \(\eta t\). Therefore, the same barrier and martingale estimates used in the proof of Theorem~\ref{thm:mainthm_existence_uniqueness} imply that
\[
\lim_{\varepsilon\downarrow0}
\mathbb P(\tau_\varepsilon\leq T)=0
\]
for every \(T>0\). The consistency of the regularised equations on the inner domains \(U_\varepsilon\) then allows the processes \(X^\varepsilon\) to be patched together, yielding a global strong solution of the original SDE. Pathwise uniqueness follows from local uniqueness on \(U\).

\end{proof}

We stress that the assumptions of Proposition \ref{prop:general_drift} are not merely technical artefacts but they are verified in concrete and physically relevant models. The following noteworthy example shows how  they are satisfied by the class of interacting particle systems representing the vortex flow studied in \cite{1985_Takanobu}, where existence and uniqueness of strong solutions for a system of stochastic differential equations with singular drift were established.

\begin{example}[Vortex type drift \cite{1985_Takanobu}]\label{example:vortex_flux}
Let $N\ge 2$ and let $d=2k, k\ge 1$. We consider $N$ interacting particles in \(\mathbb R^d\), with configuration $X_t=(X_t^1,\ldots,X_t^N) \in \mathbb R^{Nd}$, and define the collision--free configuration space 
\begin{equation}\label{def:U}
    U := \mathbb{R}^{Nd} \backslash \left\{\bigcup_{1\leq j<k\leq N} \left\{x = (x^{(1)},...,x^{(N)})\in \mathbb{R}^{Nd} : x^{(j)}= x^{(k)} \right\}\right\}.
\end{equation}
Consider the SDE 
 \begin{equation*}
        dX_t = b(X_t) dt + \sigma  dW_t
    \end{equation*}
with $b(X_t) = \left(b_1(X_t),\ldots,b_N(X_t)\right)$, where  
$
 b_i(X_t) = \sum_{j \neq i} \gamma_j \, (\nabla H)^\perp(X^i_t - X^j_t),
$ \(\gamma_j > 0\) for every \(j=1,\ldots,N\), \(W=(W_1, \ldots, W_N)\) is an \(Nd\)-dimensional Brownian motion and the diffusion coefficient satisfies \(\sigma = \sigma_0 I_{Nd},\, \sigma_0>0\).
Moreover,
\begin{itemize}
\item[a.]
\(H(x)=g(|x|)\), where \(g\in C^2((0,\infty))\) is bounded below and
\(
g(r)\rightarrow+\infty\) as \(r\downarrow0;
\)
\item[b.]
there exists \(C>0\) such that
\(
g''(r)+\frac{d-1}{r}g'(r)\leq C
\) for \(r>0\),
and
\(g'(r)=O(r), g''(r)=O(1)\) as \(r\to\infty.\)
\end{itemize}
Define the scalar function
\[
\Psi(x)
:=
\sum_{1\leq i<j\leq N} \gamma_i\,
\gamma_j H(x^i-x^j),
\qquad x\in U.
\]
We show that \(\Psi\) satisfies the assumptions of Proposition~\ref{prop:general_drift}.

Note that for \(v=(v_1,v_2)\in\mathbb R^k\times\mathbb R^k\), the orthogonal complement is defined as
\(v^\perp:=(v_2,-v_1)\),  that is the rotated vector.

Conditions (i) and (ii) hold by assumption. Moreover,
\(b\cdot\nabla\Psi=0\) by orthogonality, while
\[
\Delta H(x)
=
g''(|x|)
+\frac{d-1}{|x|}g'(|x|),
\]
so assumption b) yields condition (iii). Condition (iv) is immediate
since \(\sigma=\sigma_0I_{Nd}\). Finally, condition (v) follows from a
standard radial cut-off: for each \(\varepsilon>0\), one modifies \(g\)
only near the origin, keeping it bounded below and unchanged for
\(r\ge\varepsilon\). The assumptions \(g'(r)=O(r)\) and
\(g''(r)=O(1)\) ensure that the resulting \(b^\varepsilon\) and
\(\nabla\Psi^\varepsilon\) are globally Lipschitz and coincide with
\(b\) and \(\Psi\) on \(\overline{U_\varepsilon}\).

In conclusion, the example proposed in \cite{1985_Takanobu} provides an example of drift not in gradient form for which Proposition \ref{prop:general_drift} can be applied.
\end{example}

\section{Brownian particles interacting via a Lennard--Jones type force: a unique strong solution}\label{sec:Lennard--Jones} 

Let $\left(\Omega,\mathcal{F},\mathbb F=(\mathcal{F}_{t})_{t\ge 0},\mathbb{P}\right)$ be again a filtered probability space,
let $T>0$ be a fixed time horizon and $N\in \mathbb N$. Let \(W^1,\ldots,W^N\) be independent
\(d\)-dimensional Brownian motions adapted to \(\mathbb F\).

Consider \(N\) Brownian particles interacting through the
Lennard--Jones force \(F=-\nabla V\), where  $V$ is the  Lennard--Jones potential characterised by two parameters $\alpha,\beta\in \mathbb R_+,$ and   defined, with $A,B \in \mathbb R_+$, for any $x\in \mathbb R^d$, as
\begin{equation}\label{def:Lennard_Jones_potential}
    V (x) := \frac{A}{|x|^\alpha} - \frac{B}{|x|^\beta}, \qquad \alpha>\beta> 0.
\end{equation} 
The Lennard--Jones potential models both attraction and repulsion  between two  particles: the first term represents a short-range strong repulsion to prevent particle overlap, while the second term represents a weaker, attractive force at intermediate range, able to model the van der Waals forces.  The potential $V$ is bounded below, the minimum is attained in $|\bar{x}|=\left(\frac{A\alpha}{B\beta}\right)^\frac{1}{\alpha-\beta}$ and its value is
\begin{equation}\label{eq:minV}
\min V(x)= V(\bar{x})=
- \frac{\alpha-\beta}{\alpha}B\left(\frac{\beta B}{\alpha A}\right)^{\frac{\beta}{\alpha-\beta}}
 := -\delta^\prime,
\end{equation}
with $\delta^\prime >0$, since  $\alpha>\beta$.
The corresponding force $F= -\nabla V$ reads as 
 \begin{equation}\label{def:Lennard_Jones_force}
 F(x) = -\nabla V(x) = \left(\frac{\alpha A}{|x|^{\alpha+1}}- \frac{\beta B}{|x|^{\beta+1}}\right)\frac{x}{|x|}.
\end{equation}

Let us introduce the open set  $U_N\subset \mathbb R^{Nd}$ of admissible configurations
\begin{equation}\label{eq:collisionfree_space}
    U:=U_N
:=
\left\{
x=(x^1,\ldots,x^N)\in(\mathbb R^d)^N:
x^i\neq x^j \text{ for all }i\neq j
\right\}.
\end{equation}

Consider an $Nd$-dimensional stochastic process $X=\{X_t \equiv (X^1_t,...,X^N_t)\}_{t\in[0,T]}$, where 
 $X_t^i \in \mathbb{R}^d$ denotes the location of the $i$-th particle out of $N$ at time $t\in [0,T]$. The evolution of the process $X$  is described by the following first-order system of SDEs, for any $ t\in (0,T]$  
\begin{equation} \label{eq:SDE_LJ}
\begin{split}
         dX_t^i &= -\frac{1}{N}\mathop{\sum_{j=1}^N}_{j\neq i} \nabla V \left(X_t^i-X_t^j\right) dt + \sigma dW_t^i, \qquad i=1,\ldots,N,
\end{split}
\end{equation}
 subject to  an \(\mathcal{F}_0\)-measurable random initial condition $X_0=(X_0^1,...,X_0^N)$  such that
\begin{equation}\label{eq:LJ_initial_condition}
X_0\in U_N\quad\text{a.s.},\qquad \mathbb{E}\left[\left|X_0^i\right|^2 \right] < \infty, \qquad 
\mathbb{E} \left[ \sum_{i < j} \left|V(X_0^i - X_0^j) \right| \right] < \infty,
\end{equation}

The requirement on the initial positions is for them to be in $L^2(\Omega)$ and  that the initial mean interaction energy among each pair of particles is finite. A first goal is to show that the set of initial conditions \eqref{eq:LJ_initial_condition} is not empty.
 Propositions \ref{prop:initial_condition_alpha<d} and \ref{prop:initial_condition_alpha>=d} below provide sufficient conditions for \eqref{eq:LJ_initial_condition}.

\begin{proposition}[Initial conditions \(\alpha < d\)]\label{prop:initial_condition_alpha<d}
Let \( V \) be the potential \eqref{def:Lennard_Jones_potential} with  \( \alpha < d \). Let \( X_0 = (X_0^1, \ldots, X_0^N) \) be such that \( X_0^i \in L^2(\Omega) \) are independent and identically distributed (i.i.d.) random variables with common density \( \rho_0 \in L^p(\mathbb{R}^d) \), 
    for some \( p\geq  \frac{2d}{2d - \alpha} \).  Then, the random vector \( X_0 \) satisfies \eqref{eq:LJ_initial_condition}.
\end{proposition}
\begin{proof}
For \(i\neq j\)
\begin{align*}
    \mathbb{E}\left[\left|V\left(X_0^i-X_0^j\right)\right|\right] &=\int_{\mathbb{R}^{2d}} \left| \frac{A}{|x-y|^\alpha} -\frac{B}{|x-y|^\beta}   \right| \rho_0(x)\rho_0(y)dx dy\\
    &\le    A\int_{\mathbb{R}^{2d}} \frac{\rho_0(x)\rho_0(y)}{|x-y|^\alpha} dx\, dy +B\int_{\mathbb{R}^{2d}} \frac{\rho_0(x)\rho_0(y)}{|x-y|^\beta} dx\, dy.
\end{align*} 

Each of the integrals above is finite by the Hardy–Littlewood–Sobolev inequality \cite{2002_Lieb}, provided that $\rho_0 \in L^{p_1}(\mathbb{R}^d)$, $p_1= {2d}/(2d - \alpha)$ and $\rho_0 \in L^{p_2}(\mathbb{R}^d)$, $p_2= {2d}/(2d - \beta)$, respectively. Since $\rho_0$ is a probability density, if $\rho_0\in L^p(\mathbb{R}^d)$, with
\[
p \geq \max\left\{ \frac{2d}{2d - \alpha}, \frac{2d}{2d - \beta} \right\} = \frac{2d}{2d - \alpha},
\]
both the previous conditions are satisfied. As a consequence,
\[
\mathbb{E}\left[\left|V(X_0^i - X_0^j)\right|\right] \le C_{\alpha,d} \| \rho_0 \|_{L^{p_1}(\mathbb{R}^d)}^2 + C_{\beta,d} \| \rho_0 \|_{L^{p_2}(\mathbb{R}^d)}^2 < \infty
\] 

This ensures that the initial mean interaction energy is finite, and thus condition \eqref{eq:LJ_initial_condition} holds.
\end{proof}

Proposition \ref{prop:initial_condition_alpha<d} shows that, when \( \alpha < d \), independent initial particle positions are admissible provided that their common density satisfies the stated integrability condition.
The threshold \(p=\frac{2d}{2d-\alpha}\)
is the natural exponent arising from the Hardy--Littlewood--Sobolev inequality. In general, for \(p<\frac{2d}{2d-\alpha}\), \(L^p\)-integrability of \(\rho_0\) alone is not sufficient to guarantee finiteness of the interaction energy.

We can avoid any constraint upon the order of the singularity whenever we consider dependent initial locations, as shown in the next proposition.

\begin{proposition}[Initial conditions for any \(\alpha>\beta > 0\)]
\label{prop:initial_condition_alpha>=d}
Let \(V\) be the Lennard--Jones potential \eqref{def:Lennard_Jones_potential} and let
\(\mathcal U:\mathbb R^d\to[0,\infty)\) be a measurable, confining potential with 
$e^{-\mathcal U}\in L^1(\mathbb R^d)\cap L^2(\mathbb R^d).$
Assume that \( X_0 = (X_0^1, \ldots, X_0^N) \),  \( X_0^i \in L^2(\Omega) \), has as joint density the Gibbs measure with confinement
\[
\rho_0^N(x^1,\ldots,x^N)
=
C_N\exp\!\left(
-c\sum_{i<j}V(x^i-x^j)
-\sum_i\mathcal U(x^i)
\right),
\]
with $c,C_N >0$  such that $C_N$ is a normalising constant. Then \(X_0\) satisfies \eqref{eq:LJ_initial_condition}.
\end{proposition}
\begin{proof}
By exchangeability of the particle distribution,
\[
\mathbb{E}_{\rho_0^N} \left[ \sum_{i<j} |V(X_0^i - X_0^j)| \right] = \binom{N}{2} \mathbb{E}_{\rho_0^N} \left[ |V(X_0^1 - X_0^2)| \right],
\]
therefore it suffices to prove that
\[
 I:= \mathbb{E}_{\rho_0^N} \left[ |V(X_0^1 - X_0^2)| \right] < +\infty.
\]
If \(\rho^2_0\) denotes the two-particle marginal of \(\rho^N_0\), since by hypothesis $V\ge -\delta^\prime$, see \eqref{eq:minV}, and $\mathcal{U}\ge 0$, then there exists a constant \(C' > 0\) such that
\[
\rho_0^2(x,y) \leq C' e^{-c V(x - y) - \mathcal{U}(x) - \mathcal{U}(y)}.
\] 
Then,
\begin{align*}
I &= \int_{\mathbb{R}^d} \int_{\mathbb{R}^d} |V(x - y)| \rho_0^2(x,y) \, dx \, dy \\
  &\leq C' \int_{\mathbb{R}^d} \int_{\mathbb{R}^d} |V(x - y)| e^{-c V(x - y)} e^{-\mathcal{U}(x)} e^{-\mathcal{U}(y)} \, dx \, dy\\
  &\leq C^\prime \left( \int_{|x - y| \leq M} + \int_{|x - y| > M} \right) |V(x - y)| e^{-c V(x - y)} e^{-\mathcal{U}(x)} e^{-\mathcal{U}(y)} \, dx \, dy \\
  &=: I_1 + I_2.
\end{align*}
By applying the change of variables $z = x - y, w = y$ to the first integral $I_1$, we obtain
\begin{align*}
    I_1 &= C \int_{|z|\leq M} dz \, |V(z)| e^{-c V(|z|)} \int_{\mathbb{R}^d}  e^{-\mathcal{U}\left(z+y\right)} e^{-\mathcal{U}\left(y\right)} dy.
\end{align*}
By Cauchy–Schwarz inequality and translation invariance property, we get the boundedness of the internal integral as
\[
\int_{\mathbb{R}^d} e^{-\mathcal{U}\left(z+w\right)} e^{-\mathcal{U}\left(w\right)} dw \leq \left( \int_{\mathbb{R}^d} e^{-2 \mathcal{U}(w)} dw \right) = \| e^{-\mathcal{U}}\|^2_{L^2(\mathbb R^d)} < \infty,
\]
which gives:
\[
I_1 \leq C \left|\mathbb S^{d-1}\right|\int_0^M r^{d-1}|V(r)| e^{-c V(r)} dr,
\]
where $\mathbb{S}^{d-1}$ denotes the surface measure of a unit sphere in dimension $d$.
Near the origin, the integrand decays super-exponentially, so the integral is finite.

For $I_2$, given \(M > 0\), from \eqref{eq:minV} and since $e^{-\mathcal{U}}\in L^1(\mathbb R^d)$ by assumption, then
\begin{align*}
    I_2 &\leq \int_{|x-y|> M}C \left(\frac{A}{|x-y|^\alpha}+\frac{B}{|x-y|^{\beta}} \right) e^{c\delta'} e^{-\mathcal{U}(x)} e^{-\mathcal{U}(y)} dx dy \leq \\
    &\leq   C \left(\frac{A}{M^\alpha}+\frac{B}{M^{\beta}} \right) e^{c\delta'} \int_{|x-y|> M} e^{-\mathcal{U}(x)} e^{-\mathcal{U}(y)} dx dy \leq \\
    &\leq C_{A,B,M,\delta'}\left(\int_{\mathbb{R}^d} e^{-\mathcal{U}(x)} dx\right)^2  < +\infty.
\end{align*}
Then the statement is proved. 
\end{proof}

\begin{remark}
The Gibbs measure in Proposition~\ref{prop:initial_condition_alpha>=d} is not the only initial distribution that ensures finiteness of the initial mean interaction energy \eqref{eq:LJ_initial_condition}. More generally, any Gibbs-type distribution with sufficient confinement and strong repulsion will suffice. For instance, for any \(\theta \geq \alpha\),
\[
\rho_0^N(x^1, \dots, x^N) \propto \exp\left(- \sum_{i<j} \frac{1}{|x^i - x^j|^\theta} - \sum_{i=1}^N \mathcal{U}(x^i) \right)
\]
also yields a finite expected mean interaction energy. Other examples include hard-core interaction models and parking processes.
\end{remark}

In order to prove the well-posedness of the interacting particle system, we need to state the following crucial technical lemma.

\begin{lemma}[Coercivity of the total interaction force]
\label{lemma:LJ_coercivity}
Let us consider the Lennard--Jones potential \eqref{def:Lennard_Jones_potential} and the corresponding force \eqref{def:Lennard_Jones_force}.
Then, for every fixed \(N\ge2\), there exist constants
\(c_N>0\) and \(C_N\ge0\) such that every collision-free
configuration \(x=(x^1,\ldots,x^N)\in U_N\) satisfies
\[
\sum_{i=1}^N
\left|
\sum_{j\neq i}F(x^i-x^j)
\right|^2
\ge
c_N
\sum_{1\le i<j\le N}
|F(x^i-x^j)|^2
-C_N.
\]
\end{lemma}

\begin{proof}
We separate the proof into two regimes. Near a collision, the repulsive part of the Lennard--Jones force dominates and prevents cancellation of all total forces. Away from collisions, all pairwise forces are uniformly bounded, and the corresponding contribution can be absorbed into the additive constant \(C_N\).

Set
\begin{align*}
    \mathcal T(x)
&:=
\sum_{i=1}^N
\left|
\sum_{j\neq i}F(x^i-x^j)
\right|^2,\\
L(x)&:=\max_{i<j}|F(x^i-x^j)|,\\
\mathcal P(x)
&:=
\sum_{1\le i<j\le N}|F(x^i-x^j)|^2.
\end{align*}

\smallskip
\noindent\emph{Step 1. Configurations close to collision.}

We first show that when at least two particles are very close, the total forces cannot all be small compared with the largest pairwise force. More precisely, we claim that there exist \(\delta>0\) and \(\eta>0\) such that
\begin{equation}
    \label{eq:step1}
    r(x):=\min_{i<j}|x^i-x^j|<\delta
\quad\Longrightarrow\quad
\mathcal T(x)\ge \eta L(x)^2,
\end{equation}
meaning
\[
\sum_{i=1}^N
\left|
\sum_{j\neq i}F(x^i-x^j)
\right|^2 \ge \eta \left(\max_{i<j}|F(x^i-x^j)|\right)^2.
\]
Given the claim, we can pass from the largest pairwise force to the sum of all pairwise forces. Indeed, since there are
\(
M:=\binom N2
\)
pairs, we have
\[
\mathcal P(x)
\le
M L(x)^2,
\]
which implies, by \eqref{eq:step1}, that whenever $r(x)<\delta$
\[
\mathcal T(x)
\ge
\frac{\eta}{M}\mathcal P(x).
\]

Before giving the details, we briefly describe the strategy of the argument. We argue by contradiction and consider a sequence of configurations approaching the collision set for which the total interaction forces are assumed to become negligible compared with the largest pairwise force. We first show that the maximum force \(L_n\) is exactly of order \(r_n^{-(\alpha+1)}\), which is the natural collision scale. Then,  after normalisation by \(r_n^{-(\alpha+1)}\), we show that
\[
r_n^{\alpha+1}
\sum_{j\neq i}F(x_n^i-x_n^j)
\longrightarrow0
\qquad\text{for every }i.
\]
We then rescale the configuration by \(r_n\), so that the collapsing geometry remains visible in the limit. At this scale, the attractive part of the Lennard--Jones interaction is lower order and disappears, while the repulsive part remains of order one. After isolating the particles that stay at finite rescaled distance, the contradiction assumption therefore yields a finite limit configuration of at least two distinct particles in equilibrium under a purely repulsive force. We finally show that such an equilibrium is impossible.

Suppose, by contradiction, that \eqref{eq:step1} is false. Then we can find collision-free configurations
\[
x_n=(x_n^1,\ldots,x_n^N)\in U_N
\]
such that
\[
r_n:=\min_{i<j}|x_n^i-x_n^j|\longrightarrow0
\]
while
\[
\frac{\mathcal T(x_n)}{L_n^2}\longrightarrow0,
\qquad
L_n:=\max_{i<j}|F(x_n^i-x_n^j)|,
\]
that is
\[
\frac{\sum_{i=1}^N
\left|
\sum_{j\neq i}F(x_n^i-x_n^j)
\right|^2}{\left(\max_{i<j}|F(x_n^i-x_n^j)|\right)^2} \longrightarrow 0.
\]

Thus, although some pairwise force becomes very large, the total force acting on every particle is assumed to be negligible relative to that scale. Our goal is to show that such a cancellation cannot occur near a collision.

For each \(n\), choose a pair \((i_n,j_n)\) realising the minimum distance:
\[
|x_n^{i_n}-x_n^{j_n}|=r_n.
\]
There are only finitely many possible pairs of indices. Hence one fixed pair occurs infinitely many times. Passing to the corresponding subsequence, we may assume that there exist fixed indices \(i_0\neq j_0\) such that
\[
r_n=|x_n^{i_0}-x_n^{j_0}|
\qquad\text{for every }n.
\]
The next point is to identify the natural size of the forces near the collision. Specifically, we show that  \(L_n\) is exactly of order \(r_n^{-(\alpha+1)}\).

Indeed, since
\[
F(z)
=
\alpha A\frac{z}{|z|^{\alpha+2}}
-
\beta B\frac{z}{|z|^{\beta+2}},
\]
then
\[
|F(z)|
=
\left|
\frac{\alpha A}{|z|^{\alpha+1}}
-
\frac{\beta B}{|z|^{\beta+1}}
\right|.
\]
Because \(\alpha>\beta\), the first term is the dominant one as \(|z|\to0\). Therefore, at the minimal distance \(r_n\), the natural force scale is \(r_n^{-(\alpha+1)}\).

Therefore, for all sufficiently large \(n\),
\[
L_n
\ge
|F(x_n^{i_0}-x_n^{j_0})|
\ge
\frac{\alpha A}{2}r_n^{-(\alpha+1)}.
\]
On the other hand, since every inter-particle distance is at least \(r_n\) and \(r_n\le1\) for large \(n\),
\[
|F(x_n^i-x_n^j)|
\le
(\alpha A+\beta B)r_n^{-(\alpha+1)}.
\]
Hence taking the maximum over \(i,j\) we obtain
\[
\frac{\alpha A}{2}
\le
r_n^{\alpha+1}L_n
\le
\alpha A+\beta B.
\]
Thus \(L_n\) is exactly of order \(r_n^{-(\alpha+1)}\).

This comparison allows us to rewrite the contradiction assumption on the collision scale. Since
\[
\frac{\mathcal T(x_n)}{L_n^2}\to0, \qquad \text{i.e.}\qquad
\sum_{i=1}^N \frac{
\left|
\sum_{j\neq i}F(x_n^i-x_n^j)
\right|^2}{\left(\max_{i<j}|F(x_n^i-x_n^j)|\right)^2} \longrightarrow 0.
\]
then, for every fixed \(i\),
\[
\frac1{L_n}
\left|
\sum_{j\neq i}F(x_n^i-x_n^j)
\right|
\longrightarrow0.
\]
Because \(r_n^{\alpha+1}L_n\) is bounded, it follows that
\begin{equation}
\label{eq:scaled_net_force_zero}
r_n^{\alpha+1}
\sum_{j\neq i}F(x_n^i-x_n^j)
\longrightarrow0
\qquad\text{for every }i.
\end{equation}

We now want to understand what the configuration looks like at the scale \(r_n\). In the original variables, particles separated by distances of order \(r_n\) collapse onto the same point as \(n\to\infty\), so their relative geometry is lost. We therefore focus at the collision scale. We translate so that \(x_n^{i_0}\) becomes the origin and rescale lengths by \(r_n\):
\[
y_n^i:=\frac{x_n^i-x_n^{i_0}}{r_n}.
\]
Then
\begin{equation}
\label{eq:rescaled_separation}
y_n^{i_0}=0,
\qquad
|y_n^{j_0}|=1, \qquad |y_n^i-y_n^j|\ge1
\qquad(i\neq j).
\end{equation}
Thus the closest pair stays at distance \(1\), and no two rescaled particles can collide.

At this point, some particles may remain at finite distance from the collision, whereas others may move infinitely far away in the rescaled coordinates. Since there are only finitely many particles, after passing to a further subsequence we may assume that, for every \(i\), either \((y_n^i)_n\) is bounded or
\[
|y_n^i|\longrightarrow\infty.
\]
We therefore define the set of particles that remain visible at the collision scale by
\[
I:=\{i:\ (y_n^i)_n\text{ is bounded}\}.
\]
Since
\(
y_n^{i_0}=0\) and \(|y_n^{j_0}|=1\), then \(i_0,j_0\in I\), and \(|I|\ge2\).

For each \(i\in I\), the sequence \((y_n^i)\) is bounded. Since \(I\) is finite, after one further subsequence extraction we may assume that all of them converge simultaneously:
\[
y_n^i\longrightarrow y^i\in\mathbb R^d,
\qquad i\in I.
\]
Passing to the limit in \eqref{eq:rescaled_separation}, we obtain
\[
|y^i-y^j|\ge1
\qquad
(i,j\in I,\ i\neq j),
\]
so the limiting particles remain distinct.

We now determine which forces survive in the rescaled limit. For \(z\neq0\),
\[
r_n^{\alpha+1}F(r_nz)
=
\alpha A\frac{z}{|z|^{\alpha+2}}
-
\beta B\,r_n^{\alpha-\beta}
\frac{z}{|z|^{\beta+2}}.
\]
Because \(\alpha>\beta\),
\[
r_n^{\alpha-\beta}\longrightarrow0.
\]
Hence, at the collision scale, the attractive part disappears and only the repulsive interaction remains.

Fix \(i\in I\). If \(j\in I\), then
\[
y_n^i-y_n^j\longrightarrow y^i-y^j\neq0,
\]
and therefore
\[
r_n^{\alpha+1}
F(x_n^i-x_n^j) = r_n^{\alpha+1}F(r_n(y_n^i-y_n^j))
\longrightarrow
\alpha A
\frac{y^i-y^j}{|y^i-y^j|^{\alpha+2}}.
\]
If instead \(j\notin I\), then \(y_n^i\) stays bounded while \(|y_n^j|\to\infty\), so
\[
|y_n^i-y_n^j|\longrightarrow\infty.
\]
Therefore the rescaled interaction between \(i\) and \(j\) vanishes:
\[
r_n^{\alpha+1}
F(x_n^i-x_n^j)
\longrightarrow0.
\]
Thus, when we pass to the limit in \eqref{eq:scaled_net_force_zero}, only the interactions between particles in \(I\) remain. We obtain, for every \(i\in I\),
\begin{equation}
\label{eq:repulsive_equilibrium}
\sum_{\substack{j\in I\\ j\neq i}}
\alpha A
\frac{y^i-y^j}{|y^i-y^j|^{\alpha+2}}
=0.
\end{equation}

In other words, the contradiction assumption would produce a finite limit configuration of at least two distinct particles in equilibrium under a purely repulsive force. We now show that this configuration is impossible.

Let us take the scalar product of \eqref{eq:repulsive_equilibrium} with \(y^i\) and sum over \(i\in I\). Then
\[
\begin{aligned}
0
&=
\sum_{i\in I}
y^i\cdot
\sum_{\substack{j\in I\\j\neq i}}
\alpha A
\frac{y^i-y^j}{|y^i-y^j|^{\alpha+2}}
=
\alpha A
\sum_{\substack{i<j\\i,j\in I}}
\frac{|y^i-y^j|^2}{|y^i-y^j|^{\alpha+2}}
=
\alpha A
\sum_{\substack{i<j\\i,j\in I}}
\frac1{|y^i-y^j|^\alpha}.
\end{aligned}
\]
The last quantity is strictly positive because \(|I|\ge2\). This contradiction proves the claim \eqref{eq:step1}.

\smallskip
\noindent\emph{Step 2. Configurations away from collision.}

It remains to consider configurations satisfying
\[
r(x)\ge\delta.
\]
Here there is no singular behaviour: all particles remain at least a fixed distance apart. Since \(F\) is continuous on \(\{|z|\ge\delta\}\) and \(F(z)\to0\) as \(|z|\to\infty\), and there are exactly \(M = \binom N2 \) pairs,
\[
\sum_{1\le i<j\le N}|F(x^i-x^j)|^2
\le 
\binom N2
\left(
\sup_{|z|\ge\delta}|F(z)|
\right)^2 < \infty.
\]

From Step 1 we already know that, whenever \(r(x)<\delta\),
\[
\sum_{i=1}^N
\left|
\sum_{j\neq i}F(x^i-x^j)
\right|^2
\ge
\frac{\eta}{\binom N2}
\sum_{1\le i<j\le N}
|F(x^i-x^j)|^2.
\]
We therefore set
\[
c_N:=\frac{\eta}{\binom N2}.
\]
To absorb the configurations satisfying \(r(x)\ge\delta\), define
\[
C_N
:=
c_N
\binom N2
\left(
\sup_{|z|\ge\delta}|F(z)|
\right)^2.
\]
If \(r(x)<\delta\), then Step 1 gives directly
\[
\sum_{i=1}^N
\left|
\sum_{j\neq i}F(x^i-x^j)
\right|^2
\ge
c_N
\sum_{1\le i<j\le N}
|F(x^i-x^j)|^2,
\]
and hence certainly
\[
\sum_{i=1}^N
\left|
\sum_{j\neq i}F(x^i-x^j)
\right|^2
\ge
c_N
\sum_{1\le i<j\le N}
|F(x^i-x^j)|^2
-C_N.
\]
If instead \(r(x)\ge\delta\), then
\[
\sum_{i=1}^N
\left|
\sum_{j\neq i}F(x^i-x^j)
\right|^2
\ge0,
\]
while
\[
\begin{aligned}
c_N
\sum_{1\le i<j\le N}
|F(x^i-x^j)|^2
-C_N
&\le
c_N
\binom N2
\left(
\sup_{|z|\ge\delta}|F(z)|
\right)^2
-C_N
\\
&=0.
\end{aligned}
\]
Therefore
\[
\sum_{i=1}^N
\left|
\sum_{j\neq i}F(x^i-x^j)
\right|^2
\ge
c_N
\sum_{1\le i<j\le N}
|F(x^i-x^j)|^2
-C_N
\]
also in this case.

Combining the two regimes proves the desired estimate.
\end{proof}

The next result gives the well-posedness of the system of Brownian particles, interacting by means of the pairwise Lennard--Jones force. This is the main result of the section.
\begin{theorem}[Well-posedness for the Lennard--Jones SDE]\label{thm:LJ_wellposedness}
Let $d\geq1$,  $T>0$ and $N\geq 2$ fixed.  The system \eqref{eq:SDE_LJ}-\eqref{eq:LJ_initial_condition} admits a pathwise unique global strong solution. Almost surely, no particle collision occurs in finite time.
\end{theorem}
\begin{proof}
We apply Theorem \ref{thm:mainthm_existence_uniqueness}. With this aim, following \cite{2005_krylov_rockner} we rewrite the system as a multidimensional SDE on $\mathbb{R}^{Nd}$   and we prove that the drift satisfies  Assumptions H.

Consider the open set \(U_N\) of admissible configurations defined in \eqref{eq:collisionfree_space}, and define the interaction potential $\Phi:U_N\to\mathbb{R}$ as
\begin{align*}
    \Phi(x) &:= \frac{1}{N} \sum_{1\leq i < j\leq N} V\left(x^{i}-x^{j}\right) \qquad x= (x^{1},\dots, x^{N}) \in U_N.
\end{align*}
We use the   notations $V\left(x^{i}-x^{j}\right)=:V_{ij}$ and $-\nabla V_{ij} =: F_{ij}$.

We first prove that Assumptions $H$ are satisfied.

\smallskip
\noindent\textit{Assumption $H_1$.} The boundedness from below of the potential $\Phi$ follows from \eqref{eq:minV}. 
 
 \smallskip
\noindent\textit{Assumption \(H_2\).}
Since \(V\) is bounded from below and diverges as the inter-particle distance tends to zero, while \(d(x,\partial U_N)=2^{-1/2}\min_{i<j}|x^i-x^j|\), it follows that \(\Phi(x)\to+\infty\) as \(d(x,\partial U_N)\to0\). Hence Assumption \(H_2\) holds.

\smallskip
\noindent\textit{Assumption \(H_4\).}
The diffusion matrix is the constant matrix
\(\sigma I_{Nd}\), hence it is bounded and globally Lipschitz.

 \smallskip
\noindent\textit{Assumption $H_3$.} We verify that the potential satisfies the regularity condition. Indeed, direct computations give
\[
    |\nabla\Phi|^2 = \frac{1}{N^2} \sum_i \left|\sum_{j\neq i} F_{ij} \right|^2, \qquad \Delta\Phi = \frac{2}{N}\sum_{i<j}\Delta V_{ij}.
\]
By Lemma \ref{lemma:LJ_coercivity}, we can estimate
\begin{equation}\label{estimate:LJ}
\begin{split}
    |\nabla\Phi(x)|^2 -    \frac{\sigma^2}{2}\Delta\Phi(x) 
    &\ge \frac{c_N}{N^2} \sum_{i<j}\|F_{ij}\|^2 - \frac{\sigma^2}{N}\sum_{i<j}\Delta V_{ij} - \frac{C_N}{N^2}\\
    &= \sum_{i<j}\left( \frac{c_N}{N^2}\|F_{ij}\|^2 - \frac{\sigma^2}{N}\Delta V_{ij}\right)  - \frac{C_N}{N^2}
    \end{split}
\end{equation}
where \( F_{ij} = -\nabla V(x^i - x^j) \), \( \Delta V_{ij} = \Delta V(x^i - x^j) \), and \( C_N > 0 \) is a constant depending on \( N \).

We now analyse the behaviour of each term in \eqref{estimate:LJ}, particularly near the singularity at the origin, where \(|x^i-x^j|\to0\). From the definition of the Lennard--Jones force \eqref{def:Lennard_Jones_force}, we have
\[
    |F(x)| \sim \alpha A |x|^{-(\alpha + 1)}, \qquad |F(x)|^2 \sim (\alpha A)^2 |x|^{-2(\alpha + 1)} \quad \text{as } |x| \to 0.
\]
On the other hand, the Laplacian of the Lennard--Jones potential is
\[
    \Delta V(x) = 
    \frac{\alpha A (\alpha - d + 2)}{|x|^{\alpha + 2}} - \frac{\beta B (\beta - d + 2)}{|x|^{\beta + 2}},
\]
hence, its most singular contribution is of order at most \(|x|^{-(\alpha+2)}.\)
Since, being \(\alpha>0\), we have
\(
    2(\alpha+1)>\alpha+2,
\)
then the positive term involving \(|F(x)|^2\) is more singular than the Laplacian term as \(|x|\to0\). Consequently, for
\[
\eta_{ij}
:=
\frac{c_N}{N^2}|F_{ij}|^2
-
\frac{\sigma^2}{N}\Delta V_{ij},
\]
we have
\[
    \eta_{ij}\longrightarrow+\infty
    \qquad\text{as }|x^i-x^j|\to0.
\]
It follows that there exists \(r^*>0\) such that \(\eta_{ij}\) is bounded from below whenever
\[
0<|x^i-x^j|\leq r^*.
\]
For \(|x^i-x^j|\geq r^*\), the functions \(F\) and \(\Delta V\) are continuous and bounded, and therefore \(\eta_{ij}\) is again uniformly bounded from below. Hence there exists a constant \(C>0\) such that
\[
    \eta_{ij}\geq-C
    \qquad\text{for every }i<j.
\]
Combining this estimate with \eqref{estimate:LJ}, we conclude that
\[
    |\nabla\Phi(x)|^2
    -
    \frac{\sigma^2}{2}\Delta\Phi(x)
    \geq -C,
\]
up to a redefinition of the constant \(C\). This verifies Assumption \(H_3\).

 \noindent\textit{Assumption \(R\).}
Let \(V_{\rho_\varepsilon}\) be the regularization from Appendix~\ref{app:polynomialregularisation}, with \(\rho_\varepsilon:=\min\{\varepsilon,r_*/2\}\), and set \[\Phi^\varepsilon(x):=N^{-1}\sum_{i<j}V_{\rho_\varepsilon}(x^i-x^j).\] Since the pairwise regularization preserves the lower bound and has globally Lipschitz gradient, \(R_2\) and \(R_3\) follow immediately. Moreover, \(x\in\overline{U_\varepsilon}\) implies \(|x^i-x^j|\ge\sqrt2\,\varepsilon>\rho_\varepsilon\) for every \(i\neq j\), hence \(\Phi^\varepsilon=\Phi\) on \(\overline{U_\varepsilon}\), proving \(R_1\).
\end{proof}

\begin{remark}
    An additional advantage of the proposed approach is that the well-posedness result remains applicable when the Lennard--Jones drift is supplemented by a non-gradient singular perturbation \(\mu\), provided that Assumption \(H_3\) continues to hold. 
    The estimates above show that the squared Lennard--Jones force dominates the Laplacian contribution near the collision set. Hence perturbations for which the positive part of \(\mu\cdot\nabla\Phi\) has strictly lower-order growth near collisions, together with suitable control away from the collision set, can be included.
    In particular, rotational perturbations satisfying
$$ \mu(x)\cdot\nabla\Phi(x)=0 $$
fall within this framework, in analogy with the orthogonality mechanism used in the vortex-flow case of Example~\ref{example:vortex_flux}, but the condition is satisfied more generally by a wider class of non-gradient drifts.
\end{remark}

\subsection*{Acknowledgments}
The research is carried out within the research project PON 2021(DM 1061, DM 1062) ``Deterministic and stochastic mathematical modelling and data analysis within the study for the indoor and outdoor impact of the climate and environmental changes for the degradation of the Cultural Heritage" of the Università degli Studi di Milano. The authors are members of GNAMPA (Gruppo Nazionale per l’Analisi Matematica, la Probabilità e le loro Applicazioni) of the Italian Istituto Nazionale di Alta Matematica (INdAM).

\appendix

\section{A regularization of the Lennard--Jones interaction}
\label{app:polynomialregularisation}

Consider the Lennard--Jones-type potential \eqref{def:Lennard_Jones_potential}
and the associated force \eqref{def:Lennard_Jones_force}. We introduce a regularization that coincides with the original interaction away from a neighbourhood of the origin and preserves the lower bound of the potential.

Let \(r_*\) be the unique minimizer of \(V\),
\begin{equation}\label{eq:r_minimizer}
    r_*
=
\left(
\frac{\alpha A}{\beta B}
\right)^{\frac1{\alpha-\beta}}.
\end{equation}
Fix \(\varepsilon>0\) and introduce
\[
P_\varepsilon:=V(\varepsilon),
\qquad
D_\varepsilon:=\varepsilon V'(\varepsilon),
\qquad
E_\varepsilon:=\varepsilon^2V''(\varepsilon).
\]
For \(s=r/\varepsilon\), define the cubic polynomial
\begin{equation}
\label{eq:HermitePolynomial}
p_\varepsilon(r)=
P_\varepsilon
-\frac{2}{3}D_\varepsilon
+\frac{1}{6}E_\varepsilon
+
\left(
D_\varepsilon-\frac12E_\varepsilon
\right)s^2
+
\frac13
\left(
E_\varepsilon-D_\varepsilon
\right)s^3.
\end{equation}
It satisfies
\begin{equation}
\label{eq:HermiteConditions}
p_\varepsilon'(0)=0,
\qquad
p_\varepsilon(\varepsilon)=V(\varepsilon),
\qquad
p_\varepsilon'(\varepsilon)=V'(\varepsilon),
\qquad
p_\varepsilon''(\varepsilon)=V''(\varepsilon).
\end{equation}
We then set
\begin{equation*}
V_\varepsilon(x)
:=
\begin{cases}
p_\varepsilon(|x|),
& |x|<\varepsilon,\\[1mm]
V(|x|),
& |x|\geq\varepsilon,
\end{cases}
\qquad
\mathbf F_\varepsilon:=-\nabla V_\varepsilon.
\end{equation*}

\begin{proposition}
For every \(\varepsilon>0\),
\[
V_\varepsilon\in C^2(\mathbb R^d),
\qquad
\mathbf F_\varepsilon\in C_b^1(\mathbb R^d;\mathbb R^d).
\]
Hence \(\mathbf F_\varepsilon\) is globally Lipschitz. Moreover,
\[
V_\varepsilon(x)=V(|x|),
\qquad
\mathbf F_\varepsilon(x)=\mathbf F(x),
\qquad |x|\geq\varepsilon,
\]
and the regularization preserves the lower bound of the original potential for \(0<\varepsilon\leq r_*\), where $r_*$ was defined in \eqref{eq:r_minimizer}.
\end{proposition}

\begin{proof}
Differentiating \eqref{eq:HermitePolynomial} gives
\[
p_\varepsilon'(r)
=
\frac{2D_\varepsilon-E_\varepsilon}{\varepsilon^2}\,r
+
\frac{E_\varepsilon-D_\varepsilon}{\varepsilon^3}\,r^2.
\]
Therefore, for \(0<|x|<\varepsilon\),
\begin{equation*}
\mathbf F_\varepsilon(x)= -
\left[
\frac{2D_\varepsilon-E_\varepsilon}{\varepsilon^2}
+
\frac{E_\varepsilon-D_\varepsilon}{\varepsilon^3}|x|
\right]x.
\end{equation*}
Thus \(\mathbf F_\varepsilon(x)=O(|x|)\) as \(x\to0\), while
$$ D\mathbf F_\varepsilon(x) \longrightarrow -\frac{2D_\varepsilon-E_\varepsilon}{\varepsilon^2}I. $$
Hence the force extends as a \(C^1\) vector field at the origin. The matching conditions \eqref{eq:HermiteConditions} give \(C^2\)-matching of the potential at \(r=\varepsilon\), and therefore
$$ V_\varepsilon\in C^2(\mathbb R^d), \qquad \mathbf F_\varepsilon\in C^1(\mathbb R^d;\mathbb R^d). $$
Boundedness of \(D\mathbf F_\varepsilon\) on \(\{|x|\le\varepsilon\}\), together with agreement with the original force on \(\{|x|\ge\varepsilon\}\), yields
$$ \mathbf F_\varepsilon\in C_b^1(\mathbb R^d;\mathbb R^d). $$

For the Lennard--Jones-type potential considered here,
\[
D_\varepsilon
=
-\frac{\alpha A}{\varepsilon^\alpha}
+
\frac{\beta B}{\varepsilon^\beta}, \qquad E_\varepsilon
=
\frac{\alpha(\alpha+1)A}{\varepsilon^\alpha}
-
\frac{\beta(\beta+1)B}{\varepsilon^\beta}.
\]
Finally, we show that the regularization preserves the lower bound of the potential. Consider \(0<\varepsilon\leq r_*.\) 
For \(s=r/\varepsilon\in[0,1]\), define
$$
q_\varepsilon(s)
:=
(2D_\varepsilon-E_\varepsilon)
+
(E_\varepsilon-D_\varepsilon)s.
$$
Then
\[
p_\varepsilon'(r) = \frac r{\varepsilon^2} q_\varepsilon\!\left(\frac r\varepsilon\right). \]
Since \(q_\varepsilon\) is affine, it is sufficient to evaluate it at the endpoints.
For \(\varepsilon\le r_*\),
$$ q_\varepsilon(1) = D_\varepsilon \le0, $$
while
$$ q_\varepsilon(0) = 2D_\varepsilon-E_\varepsilon = -\frac{\alpha(\alpha+3)A}{\varepsilon^\alpha} + \frac{\beta(\beta+3)B}{\varepsilon^\beta} <0, $$
where the last inequality follows from \(\alpha>\beta\) and \(\varepsilon\le r_*\). 
In fact,
$$
\varepsilon^{\alpha-\beta}
\leq
\frac{\alpha A}{\beta B}
<
\frac{\alpha(\alpha+3)A}
{\beta(\beta+3)B}.
$$
Hence \(q_\varepsilon\le0\) on \([0,1]\), and therefore \(p_\varepsilon'\le0\) on \([0,\varepsilon]\). Consequently,
$$ p_\varepsilon(r) \ge p_\varepsilon(\varepsilon) = V(\varepsilon) \ge V(r_*). $$
For \(r\ge\varepsilon\), the same bound follows from \(V_\varepsilon(r)=V(r)\).

Consequently,
$$
V_\varepsilon(x)
\geq
V(r_*)
=
\min_{r>0}V(r)
\qquad
\text{for every }x\in\mathbb R^d.
$$
\end{proof}

\section*{Acknowledgments}
The research is carried out within the research project PON 2021(DM 1061, DM 1062) ``Deterministic and stochastic mathematical modelling and data analysis within the study for the indoor and outdoor impact of the climate and environmental changes for the degradation of the Cultural Heritage" of the Università degli Studi di Milano. The authors are members of GNAMPA (Gruppo Nazionale per l’Analisi Matematica, la Probabilità e le loro Applicazioni) of the Italian Istituto Nazionale di Alta Matematica (INdAM).

\printbibliography

@incollection{Rui_MACH_2021,
  author     = {Morale, D. and Rui, G. and Ugolini, S.},
  title      = {A stochastic interacting particle model for the marble sulphation process},
  booktitle  = {Mathematical Modeling in Cultural Heritage},
  editor     = {Bretti, G. and Cavaterra, C. and Solci, M. and Spagnuolo, M.},
  series     = {Springer INdAM Series},
  volume     = {65},
  pages      = {17--34},
  year       = {2025},
  publisher  = {Springer}
}

@article{javergaard2026hybrid,
  author     = {J{\"a}verg{\aa}rd, N. and Morale, D. and Muntean, A. and Rui, G. and Ugolini, S.},
  title      = {A hybrid model of sulphation reactions: stochastic particles in a random continuum environment},
  journal    = {Applied Mathematics in Science and Engineering},
  volume     = {34},
  number     = {1},
  pages      = {2691359},
  year       = {2026},
  publisher  = {Taylor \& Francis}
}

@article{2025_GM,
  author     = {Greco, E. M. and Morale, D.},
  title      = {Integrability properties and stochastic {McKean--Vlasov} dynamics with singular {Lennard--Jones} drift: a mesoscale regularization},
  journal    = {arXiv preprint},
  year       = {2026}
}

@article{1985_Takanobu,
  author  = {Takanobu, S.},
  title   = {On the existence and uniqueness of {SDE} describing an $n$-particle system interacting via a singular potential},
  journal = {Proceedings of the Japan Academy, Series A, Mathematical Sciences},
  volume  = {61},
  number  = {9},
  pages   = {287--290},
  year    = {1985},
}

@article{2001_Gyongy_martinez,
  author     = {Gy{\"o}ngy, I. and Mart{\'i}nez, T.},
  title      = {On stochastic differential equations with locally unbounded drift},
  journal    = {Czechoslovak Mathematical Journal},
  volume     = {51},
  number     = {4},
  pages      = {763--783},
  year       = {2001}
}

@book{2002_Lieb,
  author     = {Lieb, E. H. and Loss, M.},
  title      = {Analysis},
  edition    = {2},
  series     = {Graduate Studies in Mathematics},
  volume     = {14},
  publisher  = {American Mathematical Society},
  address    = {Providence, RI},
  year       = {2001}
}

@incollection{2003_albeverio_kondratiev_Rockner,
  author     = {Albeverio, S. and Kondratiev, Yu. G. and R{\"o}ckner, M.},
  title      = {Strong Feller properties for distorted Brownian motion and applications to finite particle systems with singular interactions},
  booktitle  = {Finite and Infinite Dimensional Analysis in Honor of Leonard Gross},
  series     = {Contemporary Mathematics},
  volume     = {317},
  pages      = {15--35},
  year       = {2003},
  publisher  = {American Mathematical Society},
  address    = {Providence, RI}
}

@article{2005_Zhang,
  author     = {Zhang, X.},
  title      = {Strong solutions of {SDEs} with singular drift and Sobolev diffusion coefficients},
  journal    = {Stochastic Processes and their Applications},
  volume     = {115},
  number     = {11},
  pages      = {1805--1818},
  year       = {2005}
}

@book{2005_cherny,
  author     = {Cherny, A. S. and Engelbert, H.-J.},
  title      = {Singular stochastic differential equations},
  series     = {Lecture Notes in Mathematics},
  volume     = {1858},
  edition    = {1},
  publisher  = {Springer},
  address    = {Berlin, Heidelberg},
  year       = {2005}
}

@article{2005_krylov_rockner,
  author     = {Krylov, N. V. and R{\"o}ckner, M.},
  title      = {Strong solutions of stochastic equations with singular time-dependent drift},
  journal    = {Probability Theory and Related Fields},
  volume     = {131},
  number     = {2},
  pages      = {154--196},
  year       = {2005}
}

@article{2014_Beck_Flandoli_Gubinelli_Maurelli,
  author     = {Beck, L. and Flandoli, F. and Gubinelli, M. and Maurelli, M.},
  title      = {Stochastic {ODEs} and stochastic linear {PDEs} with critical drift: regularity, duality and uniqueness},
  journal    = {Electronic Journal of Probability},
  volume     = {24},
  pages      = {1--72},
  year       = {2019}
}

@article{2015_Godinho_quininao,
  author     = {Godinho, D. and Qui{\~n}inao, C.},
  title      = {Propagation of chaos for a subcritical {Keller--Segel} model},
  journal    = {Annales de l'Institut Henri Poincar{\'e}, Probabilit{\'e}s et Statistiques},
  volume     = {51},
  number     = {3},
  pages      = {965--992},
  year       = {2015}
}

@article{Wales,
  author     = {Schwerdtfeger, P. and Wales, D. J.},
  title      = {100 years of the {Lennard--Jones} potential},
  journal    = {Journal of Chemical Theory and Computation},
  volume     = {20},
  number     = {9},
  pages      = {3379--3405},
  year       = {2024}
}

@book{RogersWilliams,
 author     = {Rogers, L. C. G. and Williams, D.},
  title      = {Diffusions, Markov processes, and martingales: volume 1, foundations},
  edition    = {2},
  series     = {Cambridge Mathematical Library},
  publisher  = {Cambridge University Press},
  year       = {2000}
}

@article{2016_Liu_Yang,
  author     = {Liu, J.-G. and Yang, R.},
  title      = {Propagation of chaos for large Brownian particle system with Coulomb interaction},
  journal    = {Research in the Mathematical Sciences},
  volume     = {3},
  pages      = {40},
  year       = {2016}
}

@article{2016_MZCJ,
  author     = {Morale, D. and Zanella, M. and Capasso, V. and J{\"a}ger, W.},
  title      = {Stochastic modeling and simulation of ion transport through channels},
  journal    = {Multiscale Modeling \& Simulation},
  volume     = {14},
  number     = {1},
  pages      = {113--137},
  year       = {2016}
}

@article{2018_Dai_Pra,
  author     = {Andreis, L. and Dai Pra, P. and Fischer, M.},
  title      = {{McKean--Vlasov} limit for interacting systems with simultaneous jumps},
  journal    = {Stochastic Analysis and Applications},
  volume     = {36},
  number     = {6},
  pages      = {960--995},
  year       = {2018},
  publisher  = {Taylor \& Francis}
}

@article{2019_Liu_Yang,
  author     = {Liu, J.-G. and Yang, R.},
  title      = {Propagation of chaos for the {Keller--Segel} equation with a logarithmic cut-off},
  journal    = {Methods and Applications of Analysis},
  volume     = {26},
  number     = {4},
  pages      = {319--348},
  year       = {2019}
}

@article{2020_Flandoli_Leocata_Ricci,
  author     = {Flandoli, F. and Leocata, M. and Ricci, C.},
  title      = {On the macroscopic limit of Brownian particles with local interaction},
  journal    = {Stochastics and Dynamics},
  volume     = {20},
  number     = {6},
  pages      = {2040007},
  year       = {2020}
}

@article{2020_Nam,
  author     = {Nam, K.},
  title      = {Stochastic differential equations with critical drifts},
  journal    = {Stochastic Processes and their Applications},
  volume     = {130},
  number     = {9},
  pages      = {5366--5393},
  year       = {2020}
}

@article{2021_Kinzebulatov_Semenov,
  author     = {Kinzebulatov, D. and Semenov, Y. A.},
  title      = {Sharp solvability for singular {SDEs}},
  journal    = {Electronic Journal of Probability},
  volume     = {28},
  pages      = {1--15},
  year       = {2023}
}

@article{2023_Yang_Zhang,
  author     = {Yang, S. and Zhang, T.},
  title      = {Strong solutions to reflecting stochastic differential equations with singular drift},
  journal    = {Stochastic Processes and their Applications},
  volume     = {156},
  pages      = {126--155},
  year       = {2023}
}

\end{document}